 \documentclass{conm-p-l}
\usepackage{amssymb}

\newtheorem{theorem}{Theorem}[section]

\theoremstyle{definition}
\newtheorem{definition}[theorem]{Definition}

\theoremstyle{remark}
\newtheorem{remark}[theorem]{Remark}
\newtheorem{coro}[theorem]{Corollary}

\numberwithin{equation}{section}

\begin{document}

\title[Sparse Operators and their boundedness on Morrey-type Spaces]{Sparse Operators and their boundedness on Morrey-type Spaces: An Expository Note}
 
\author{Manasa N. Vempati}
\address{Department of Mathematics\\ Louisiana State University\\ Baton Rouge \\ LA 70803 \\ USA}
\email{nvempati@lsu.edu}

\date{}
\dedicatory{This paper is dedicated in honor of Professor Boris Rubin's 80th birthday.}
\begin{abstract}
Sparse domination is a central tool in modern harmonic analysis, offering a unified approach to weighted inequalities for Calder\'on--Zygmund operators and related operators such as commutator operators, rough singular integrals, square functions etc. In this expository note, we briefly survey the main ideas behind sparse bounds on $L^p$ and weighted $L^p$ spaces, and discuss boundedness results for sparse operators on Morrey and generalized Morrey spaces. We establish a model result on the boundedness of sparse operators on Komori--Shirai type weighted Morrey spaces $L^{p,\kappa}(w)$, $w\in A_p$. These bounds offer a simpler framework for deriving Morrey-space estimates 
for Calder\'on--Zygmund operators and related operators via sparse domination. We conclude with remarks about extensions and connections to other Morrey space settings.
\end{abstract}

\maketitle

\section{Introduction\label{Section:Introduction}}

The classical Morrey spaces were introduced by Morrey in 1938 in his study of quasi-linear elliptic equations~\cite{Morrey}. For $1 \le p < \infty$ and $0 \le \lambda \le n$, the Morrey space $L^{p,\lambda}(\mathbb{R}^{n})$ consists of all locally integrable functions $f$ such that
\[
\|f\|_{L^{p,\lambda}(\mathbb{R}^n)}
:= \sup_{B(x,r)} r^{-\lambda/p}
\bigg( \int_{B(x,r)} |f(y)|^{p} \, dy \bigg)^{1/p} < \infty,
\]
where the supremum is taken over all balls $B(x,r)\subset \mathbb{R}^{n}$. 

Morrey spaces generalize the classical Lebesgue spaces and exhibit finer analytic properties than these spaces, by incorporating a more delicate control on the local behaviour of functions. In particular, for $\lambda=0$ and $\lambda=n$, the Morrey spaces $L^{p,0}(\mathbb{R}^{n})$ and $L^{p,n}(\mathbb{R}^{n})$ coincide with the Lebesgue spaces $L^{p}(\mathbb{R}^{n})$ and $L^{\infty}(\mathbb{R}^{n})$, respectively. Thus Morrey spaces extend the classical $L^{p}$ spaces while offering finer control on local behaviour, which makes them suitable for capturing regularity phenomena in PDE and harmonic analysis. Moreover, Morrey spaces form complete normed spaces and admit several equivalent formulations. 

Since their introduction, many results originally established for Lebesgue spaces have been extended to Morrey spaces. Adams~\cite{Adams} proved the boundedness of Riesz potentials on Morrey spaces, and Chiarenza-Frasca~\cite{FM} later showed that the Hardy-Littlewood maximal operator is bounded on $L^{p,\lambda}(\mathbb{R}^{n})$ for $1<p<\infty$ and $0<\lambda<n$. It was subsequently established that the CalderA3n-Zygmund operators act boundedly on $L^{p,\lambda}(\mathbb{R}^{n})$ for $0 \le \lambda < n$. In addition, Olsen~\cite{Olsen} proved a HAlder-type inequality in Morrey spaces: if $f \in L^{p,\lambda}(\mathbb{R}^n)$ and $g \in L^{q,\mu}(\mathbb{R}^n)$ with $1 \le p, q < \infty$, $\frac{1}{p} + \frac{1}{q} \ge 1$, and
\[
\frac{1}{r} = \frac{1}{p} + \frac{1}{q},
\qquad 
\frac{\nu}{r} = \frac{\lambda}{p} + \frac{\mu}{q},
\]
then
\[
\|fg\|_{L^{r,\nu}(\mathbb{R}^{n})} 
\le \|f\|_{L^{p,\lambda}(\mathbb{R}^{n})}
   \|g\|_{L^{q,\mu}(\mathbb{R}^{n})}.
\]

Weighted Morrey spaces further refine this framework by simultaneously encoding local behavior and non-uniform density through weights, thereby generalizing both unweighted Morrey spaces and weighted Lebesgue spaces $L^{p}(w)$. This makes them particularly natural for studying PDEs, as well as for the study of operators whose behavior depends on both the underlying geometry and local integrability. In view of the rapidly developing weighted theory on $L^{p}(w)$ spaces, it is therefore natural to investigate the boundedness of fundamental operators, such as the Hardy-Littlewood maximal operator, Calder\'on-Zygmund singular integrals, fractional integrals, and their commutators on weighted Morrey spaces and their generalizations. Representative results in this direction can be found in the literature on weighted and generalized Morrey spaces~\cite{gvqx,jm, akl, rjem, vsg} and in more recent developments such as~\cite{mnv}.

At the same time note that, the last decade has seen the rise of sparse domination as a powerful framework for singular integrals and related operators. The basic idea, going back to work of Lerner and subsequent developments~\cite{LernerLocal,LernerA2,LernerOmbrosi,Barron,LaceySparse,BaileyEtAl}, is that a large class of linear and multilinear operators can be dominated pointwise by positive sparse operators built from averages over a carefully chosen collection of cubes. This allows one to deduce sharp weighted estimates, endpoint bounds, and quantitative dependence on Muckenhoupt characteristics in a conceptually unified way.

A vast part of the sparse literature has focused on $L^p$ and weighted $L^p(w)$ spaces. By now there are very refined results, for example sharp $A_p$ and $A_{p,q}$ bounds for Calder\'on--Zygmund and fractional integral operators; see for example \cite{LernerA2,LernerOmbrosi,Barron,LaceySparse}. On the other hand, Morrey spaces and their weighted/generalized variants originally studied by Komori and Shirai~\cite{KomoriShirai} and many others \cite{AlmeidaSamko2009, Nakai1994} have received less attention from the perspective of sparse operators. 

\medskip

The purpose of this note is twofold:
\begin{itemize}
  \item To give a concise overview of sparse bounds on $L^p$ and $L^p(w)$ spaces along with a brief literature guide;
  \item To present a self-contained proof of a model result: the boundedness of weighted sparse operators on the Komori--Shirai weighted Morrey spaces $L^{p,\kappa}(w)$ with $w \in A_p$ and $0 \le \kappa < 1$.
\end{itemize}
This shows how the sparse framework naturally extends to Morrey-type norms and provides a blueprint for further generalizations such as fractional operators, generalized Morrey, spaces of homogeneous type, Schr\"odinger settings, and so on. We remark that there are several classical results that already establish the boundedness of Calder\'on-Zygmund operators, fractional integrals, and related operators on weighted Morrey spaces; see, for example, the work of Komori-Shirai \cite{KomoriShirai} and its various extensions \cite{GuliyevOM, SawanoOM}. These results, however, rely on traditional tools such as sharp maximal functions, covering lemmas, and Hedberg-type inequalities, and do not make use of sparse domination techniques. In contrast, our aim in this note is to present an alternative approach based on sparse bounds. By working directly with sparse averaging operators and their weighted estimates, we obtain a simpler proof to Morrey-type inequalities. This approach not only simplifies several arguments, but also lays the groundwork for incorporating sparse methods into the weighted Morrey framework.

\medskip

In this note, we will work throughout on $\mathbb{R}^n$ equipped with Lebesgue measure and the constants can differ from line to line typically depending on $n$ and weight constants. Also note that, in this article we assume $w\in A_p$, which is not always necessary for the boundedness of operators on weighted Morrey spaces. In particular, for certain fractional-type operators, boundedness can be obtained under weaker conditions on the weight; see, for example, \cite{NakamuraSawanoTanaka2018}. However, in the present paper, we restrict ourselves to the $A_p$ framework, which is natural in the context of sparse domination and sufficient for the applications considered here.


\section{Sparse operators on $L^p$ and weighted $L^p$ Spaces}

\subsection{Sparse families and model operators}

We will now briefly recall the standard definitions. Let $\mathcal{D}$ denote the dyadic cubes in $\mathbb{R}^n$, i.e.
\[
  \mathcal{D} := \big\{ 2^{-k}([0,1)^n + m) : k \in \mathbb{Z},\ m \in \mathbb{Z}^n \big\}.
\]

\begin{definition}
Let $0 < \eta \le 1$. A collection $\mathcal{S} \subset \mathcal{D}$ is called \emph{$\eta$-sparse} if for every $Q \in \mathcal{S}$ there exists a measurable set $E_Q \subset Q$ such that $|E_Q| \ge \eta |Q|$ and the sets $(E_Q)_{Q\in\mathcal{S}}$ are pairwise disjoint.

\end{definition}

Note that we have that Sparse families satisfy a Carleson packing condition: for any cube $R$,
\[
  \sum_{\substack{Q\in\mathcal{S}\\Q\subset R}} |Q| \lesssim |R|,
\]
with constants depending only on $\eta$ and $n$.  

Now let us introduce $r$-sparse operator associated to a sparse family.

\begin{definition}
Let $1 \le r < \infty$ and $\mathcal{S}$ be a sparse family. The associated $r$-sparse operator is
\[
  \mathcal{A}_{\mathcal{S}}^{(r)}f(x)
  := \sum_{Q\in\mathcal{S}} \bigg( \frac{1}{|Q|}\int_Q |f(y)|^r\,dy \bigg)^{1/r} \mathbf{1}_Q(x).
\]
\end{definition}

For many applications, the case $r=1$ suffices. One can also define weighted versions with averages against a weight $w$:
\[
  \mathcal{A}_{\mathcal{S},w}^{(r)}f(x)
  := \sum_{Q\in\mathcal{S}} \bigg( \frac{1}{w(Q)}\int_Q |f(y)|^r w(y)\,dy \bigg)^{1/r} \mathbf{1}_Q(x).
\]


A principal motivation for the study of sparse operators is the following kind of pointwise control, that is, given a Calder\'on--Zygmund operator $T$, there exist finitely many sparse families $\mathcal{S}_1,\dots,\mathcal{S}_N$ such that
\begin{equation}\label{eq:sparse_dom}
  |Tf(x)| \lesssim \sum_{j=1}^N \mathcal{A}^{(1)}_{\mathcal{S}_j} f(x)
  \quad\text{for a.e. }x\in\mathbb{R}^n
\end{equation}
for all compactly supported $f$. 

This type of domination was first obtained in the work of Lerner~\cite{LernerLocal,LernerA2}, and then extended and refined by many authors; see, for instance, Lerner--Ombrosi~\cite{LernerOmbrosi}, Lacey~\cite{LaceySparse}, Barron~\cite{Barron}, and more recent works such as Bailey-Bennett-Bernicot et al.~\cite{BaileyEtAl}. There are also several variants that exist now for for rough singular integrals, oscillatory integrals, square functions, variation norms, and multilinear operators \cite{CondeAlonsoRey, HytTalOscillatory, BernicotFreyPetermichl, LernerVariation, BernicotBernicot}.


Let us recall how sparse operators behave on weighted Lebesgue spaces. For this we define Muckenhoupt class $A_p$ below;

\subsection{Muckenhoupt $A_p$ Weights} Let $\omega(x)$ be a nonnegative locally integrable function
  on~$\mathbb{R}^n$. For $1 < p < \infty$, we
  say $\omega$ is an $A_p(X)$ \emph{weight}, written $\omega\in
  A_p$, if
  \begin{equation*}
    [\omega]_{A_p}
    := \sup_B \left(\frac{1}{|B|}\int_B \omega(x)dx\right)
    \left(\frac{1}{|B|}\int_B \omega(x)^{1-p'}dx\right)^{p-1}
    < \infty.
  \end{equation*}
  Here the suprema are taken over all balls~$B\subset X$.
  The quantity $[\omega]_{A_p}$ is called the \emph{$A_p$~constant
  of~$\omega$}.

\vspace{0.2 cm}
For $p = 1$, we say $\omega$ is an $A_1(X)$ \emph{weight},
  written $\omega\in A_1$, if $M(\omega)(x)\leq \omega(x)$ for $\mu$-almost every $x\in X$ and we define $A_{\infty}(X) = \cup_{1<p<\infty} A_p(X)$.

\vspace{0.2 cm}

We now define $A_{p,q}(X)$ weights (see, \cite{KomoriShirai}). We say a weight function $\omega$ belongs to $A_{p,q}$, for $1<p<q<\infty,$ if there exists a constant $C>1$ such that

 \begin{equation*}
    [\omega]_{A_{p,q}}
    := \sup_B \left(\frac{1}{|B|}\int_B \omega(x)^q d\mu(x)\right)^{1/q}
    \left(\frac{1}{|B|}\int_B \omega(x)^{-p'}dx\right)^{1/p'}
    < \infty
  \end{equation*}
  where $1/p+1/p'=1$.

  When $p=1$, $\omega\in A_{1,q}$ with $1<q<\infty$ if there exists $C>1$ such that 

  \begin{equation*}
    [\omega]_{A_{p,q}}
    := \sup_B \left(\frac{1}{|B|}\int_B \omega(x)^q dx\right)
    \left(\operatorname{esssup}_{x\in B}\frac{1}{\omega(x)}\right)^{1/p'}
    < \infty
  \end{equation*}

It is well known that $A_p$ weights are doubling and satisfy various reverse H\"older inequalities.

We know the following boundedness result for the sparse operators;

\begin{theorem}\label{thm:sparseLp}
Let $1<p<\infty$ and $w\in A_p$. For any sparse family $\mathcal{S}$,
\[
  \|\mathcal{A}^{(1)}_{\mathcal{S}} f\|_{L^p(w)}
  \lesssim [w]_{A_p}^{\max\{1,\frac{1}{p-1}\}} \|f\|_{L^p(w)} \quad\text{for all }f\in L^p(w).
\]
\end{theorem}

This is now standard; see, for example, Lerner~\cite{LernerA2}, Lacey~\cite{LaceySparse}, Lerner--Ombrosi~\cite{LernerOmbrosi}, and Barron~\cite{Barron}. In particular, combining \eqref{eq:sparse_dom} with Theorem~\ref{thm:sparseLp}, one obtains sharp weighted bounds for Calder\'on--Zygmund operators:
\[
  \|Tf\|_{L^p(w)} \lesssim [w]_{A_p}^{\max\{1,\frac{1}{p-1}\}} \|f\|_{L^p(w)}.
\]

Analogous results hold for fractional sparse operators and lead to sharp $A_{p,q}$ bounds for fractional integrals; see, for instance, works of Lacey and others on sparse bounds for fractional integrals and oscillatory operators~\cite{LaceySparse,BaileyEtAl}.


\subsection{Weighted Morrey spaces and its variants}

We again recall classical Morrey spaces, this has another equivalent formulations. For instance, fix $1\le p<\infty$ and $0\le\lambda\le n$, the Morrey space $\mathcal{M}^{p,\lambda}(\mathbb{R}^n)$ an equivalent formulation of classical Morrey spaces in terms of the volume of balls. This formulation is standard in harmonic analysis and will be convenient for our arguments. The Morrey space $\mathcal{M}^{p,\lambda}(\mathbb{R}^n)$
consists of all locally integrable functions $f$ such that
\[
  \|f\|_{\mathcal{M}^{p,\lambda}}
  := \sup_{B} |B|^{\frac{\lambda-n}{np}}
    \bigg(\int_B |f(x)|^p\,dx\bigg)^{1/p} < \infty,
\]
where the supremum is taken over all balls $B\subset\mathbb{R}^n$.
 The definition of $M^{p,\lambda}(\mathbb{R}^n)$ given above is equivalent to the classical Morrey space $L^{p,\mu}(\mathbb{R}^n)$ defined in terms of the radius of balls, that is using $|B(x,r)| \sim r^n$, the two norms are equivalent upon identifying $\mu = n - \lambda$. 

Komori and Shirai introduced a weighted counterpart for the classical Morrey space in~\cite{KomoriShirai}. For $w$ be a weight and $0\le\kappa<1$, we will now define the weighted Morrey space below;

\begin{definition}
For $1\le p<\infty$ and $0\le\kappa<1$, the weighted Morrey space $L^{p,\kappa}(w)$ consists of all locally integrable $f$ such that
\[
  \|f\|_{L^{p,\kappa}(w)}
  := \sup_{B} \bigg( \frac{1}{w(B)^\kappa}\int_B |f(x)|^p w(x)\,dx\bigg)^{1/p} < \infty,
\]
where $w(B) := \int_B w(x)\,dx$ and the supremum is over all balls $B\subset\mathbb{R}^n$.
\end{definition}

Observe for $\kappa=0$, we recover the weighted Lebesgue space $L^p(w)$, and for $0<\kappa<1$ these are strictly intermediate spaces. Komori and Shirai studied the boundedness of classical operators Hardy-Littlewood maximal, Calder\'on-Zygmund operators, fractional integrals, etc., on these spaces, assuming $w\in A_p$ and related weight classes such as $A_{p,q}$~\cite{KomoriShirai}. There are several other variants of Morrey spaces that have been studied such as weighted local Morrey spaces~\cite{NakamuraSawano}, generalized Morrey spaces~\cite{GunawanInclusion,SawanoNote}, and there have been results concerning some classical operator bounds in those settings~\cite{GunawanBessel,WangZheng,ChenDingWang,TanWangXue,Ramadana}.


While sparse domination is commonplace in the $L^p(w)$ setting, the literature on explicit sparse operator bounds on Morrey-type spaces is comparatively sparse. Many results study operators Calder\'on--Zygmund, fractional integrals, Schr\"odinger-type operators, commutators on Morrey or generalized Morrey spaces, sometimes with weights, using classical techniques such as dyadic decompositions, Hedberg-type inequalities, and covering arguments; see, e.g., Gunawan et al.~\cite{GunawanBessel}, Tanaka~\cite{Tanaka}, and works surveyed in~\cite{SawanoBook,SawanoNote}. The recent papers on generalized (weighted) Morrey and Morrey-Campanato spaces discuss sparse domination implicitly in the analysis of operator bounds; see, for example, Ramadana et al.~\cite{Ramadana} and Tan-Wang-Xue~\cite{TanWangXue}. However, to the best of the author's knowledge, there is no systematic, self-contained treatment of the model sparse averaging operators on the Komori-Shirai spaces $L^{p,\kappa}(w)$ or its various extensions for general $A_p$ weights, stated and proved in a general form. 

In the next section we give such a result with a detailed proof. The argument is elementary once one combines the standard $L^p(w)$ boundedness of sparse operators with the structure of the Morrey norm and the doubling properties of $A_p$ weights.

\section{Sparse boundedness on weighted Morrey spaces}

We now formulate and prove the main model result.

Let $1<p<\infty$ and $w\in A_p$. Given a sparse family $\mathcal{S}$, we consider the weighted sparse operator
\begin{equation}\label{eq:weighted_sparse_def}
  \mathcal{A}_{\mathcal{S}} f(x)
  := \sum_{Q\in\mathcal{S}} \bigg( \frac{1}{w(Q)}\int_Q |f(y)|^p w(y)\,dy \bigg)^{1/p} \mathbf{1}_Q(x).
\end{equation}
Note that the averaging exponent is $p$, which is adapted to the underlying $L^p(w)$ and $L^{p,\kappa}(w)$ scales. Also note that, when $w\in A_p$, then $L^{p,\kappa}(w)\subset L^1_{\mathrm{loc}}$, and hence $\mathcal{A}_{\mathcal{S}}f$ is well
defined for $f\in L^{p,\kappa}(w)$ by truncation, as in \cite{KomoriShirai} and \cite{SawanoBook}.

\begin{theorem}\label{thm:sparseMorrey}
Let $1<p<\infty$, $0\le\kappa<1$, and $w\in A_p$. For any sparse family $\mathcal{S}$, the operator $\mathcal{A}_{\mathcal{S}}$ defined in~\eqref{eq:weighted_sparse_def} is bounded on the weighted Morrey space $L^{p,\kappa}(w)$; that is, there exists a constant $C>0$ (depending on $n,p,\kappa$ and $[w]_{A_p}$) such that
\[
  \|\mathcal{A}_{\mathcal{S}} f\|_{L^{p,\kappa}(w)}
  \le C\, \|f\|_{L^{p,\kappa}(w)}
  \quad\text{for all } f\in L^{p,\kappa}(w).
\]
\end{theorem}

\begin{proof}
By definition,
\[
  \|\mathcal{A}_{\mathcal{S}} f\|_{L^{p,\kappa}(w)}^p
  = \sup_{B} \frac{1}{w(B)^\kappa}
  \int_B |\mathcal{A}_{\mathcal{S}} f(x)|^p w(x)\,dx,
\]
where the supremum runs over all balls $B\subset\mathbb{R}^n$. Fix one such ball $B_0$ and set
\[
  I := \frac{1}{w(B_0)^\kappa} \int_{B_0} |\mathcal{A}_{\mathcal{S}} f(x)|^p w(x)\,dx.
\]
We aim to prove $I \lesssim \|f\|_{L^{p,\kappa}(w)}^p$ with a constant independent of $B_0$ and $f$.

Decompose $f$ into a local and a global part relative to $B_0$:
\[
  f = f_1 + f_2, \quad
  f_1 := f\mathbf{1}_{2B_0}, \quad
  f_2 := f\mathbf{1}_{(2B_0)^c},
\]
where $2B_0$ is the concentric ball with twice the radius. Then
\[
  \mathcal{A}_{\mathcal{S}} f
  = \mathcal{A}_{\mathcal{S}} f_1 + \mathcal{A}_{\mathcal{S}} f_2
\]
and, using $(a+b)^p \le 2^{p-1}(a^p + b^p)$,
\[
  I \lesssim I_1 + I_2,
\]
where
\[
  I_1 := \frac{1}{w(B_0)^\kappa} \int_{B_0} |\mathcal{A}_{\mathcal{S}} f_1(x)|^p w(x)\,dx,\quad
  I_2 := \frac{1}{w(B_0)^\kappa} \int_{B_0} |\mathcal{A}_{\mathcal{S}} f_2(x)|^p w(x)\,dx.
\]

\medskip\noindent
We will first focus on estimate on $I_1$.
The operator $\mathcal{A}_{\mathcal{S}}$ is bounded on $L^p(w)$ for $w\in A_p$, with norm depending on $[w]_{A_p}$; this follows from Theorem~\ref{thm:sparseLp} applied to the weighted averages (see, e.g.,~\cite{LernerA2,LaceySparse,LernerOmbrosi,Barron}). Thus
\[
  \int_{\mathbb{R}^n} |\mathcal{A}_{\mathcal{S}} f_1(x)|^p w(x)\,dx
  \lesssim \int_{\mathbb{R}^n} |f_1(x)|^p w(x)\,dx
  = \int_{2B_0} |f(x)|^p w(x)\,dx.
\]
Restricting integration to $B_0$,
\[
  \int_{B_0} |\mathcal{A}_{\mathcal{S}} f_1(x)|^p w(x)\,dx
  \le \int_{\mathbb{R}^n} |\mathcal{A}_{\mathcal{S}} f_1(x)|^p w(x)\,dx
  \lesssim \int_{2B_0} |f(x)|^p w(x)\,dx.
\]
Therefore
\[
  I_1 \lesssim \frac{1}{w(B_0)^\kappa} \int_{2B_0} |f(x)|^p w(x)\,dx.
\]
By the definition of the Morrey norm,
\[
  \int_{2B_0} |f(x)|^p w(x)\,dx
  \le \|f\|_{L^{p,\kappa}(w)}^p\, w(2B_0)^\kappa.
\]
Hence
\[
  I_1 \lesssim \|f\|_{L^{p,\kappa}(w)}^p
  \frac{w(2B_0)^\kappa}{w(B_0)^\kappa}.
\]
Since $w\in A_p\subset A_\infty$, $w$ is doubling, so $w(2B_0) \le C\,w(B_0)$. Thus
\[
  \frac{w(2B_0)^\kappa}{w(B_0)^\kappa} \le C^\kappa,
\]
and we conclude
\[
  I_1 \lesssim \|f\|_{L^{p,\kappa}(w)}^p.
\]

\medskip\noindent
\emph{Step 2: Global part $I_2$.}
We next estimate
\[
  I_2 = \frac{1}{w(B_0)^\kappa} \int_{B_0} |\mathcal{A}_{\mathcal{S}} f_2(x)|^p w(x)\,dx.
\]
For $x\in B_0$, only cubes $Q\in\mathcal{S}$ with $x\in Q$ and $Q$ intersecting $(2B_0)^c$ contribute to $\mathcal{A}_{\mathcal{S}} f_2(x)$. Such cubes must contain $B_0$ and have side length larger than that of $B_0$. For $k\ge1$, let
\[
  \mathcal{S}_k := \{ Q\in\mathcal{S} : Q\supset B_0,\ \ell(Q)\approx 2^k \ell(B_0)\}.
\]
For each $x\in B_0$, we then have
\[
  \mathcal{A}_{\mathcal{S}} f_2(x)
  \le \sum_{k=1}^\infty
  \sum_{Q\in\mathcal{S}_k}
    \bigg( \frac{1}{w(Q)}\int_Q |f_2(y)|^p w(y)\,dy \bigg)^{1/p} \mathbf{1}_Q(x).
\]
Since for fixed $k$ and $x$ there are only finitely many such cubes containing $x$ (at most one per dyadic grid), and each $Q\in\mathcal{S}_k$ is contained in a fixed multiple of $2^{k+1}B_0$, we may bound
\[
  \mathcal{A}_{\mathcal{S}} f_2(x)
  \lesssim \sum_{k=1}^\infty
  \bigg( \frac{1}{w(2^{k+1}B_0)}\int_{2^{k+1}B_0} |f(y)|^p w(y)\,dy \bigg)^{1/p}
  =: \sum_{k=1}^\infty a_k,
  \quad x\in B_0.
\]
This bound is independent of $x$ for $x\in B_0$, so
\[
  I_2 = \frac{1}{w(B_0)^\kappa}
  \int_{B_0} \Big(\sum_{k=1}^\infty a_k\Big)^p w(x)\,dx
  \lesssim \frac{w(B_0)}{w(B_0)^\kappa} \Big(\sum_{k=1}^\infty a_k\Big)^p.
\]
Using $(\sum_k a_k)^p \lesssim \sum_k a_k^p$ for non-negative $a_k$,
\[
  I_2 \lesssim \frac{w(B_0)}{w(B_0)^\kappa} \sum_{k=1}^\infty a_k^p
  = \frac{w(B_0)}{w(B_0)^\kappa} \sum_{k=1}^\infty
     \frac{1}{w(2^{k+1}B_0)}\int_{2^{k+1}B_0} |f(y)|^p w(y)\,dy.
\]
By the Morrey norm,
\[
  \int_{2^{k+1}B_0} |f(y)|^p w(y)\,dy
  \le \|f\|_{L^{p,\kappa}(w)}^p\, w(2^{k+1}B_0)^\kappa,
\]
so
\[
  a_k^p
  \le \|f\|_{L^{p,\kappa}(w)}^p\, w(2^{k+1}B_0)^{\kappa-1}.
\]
Thus
\[
  I_2 \lesssim \frac{w(B_0)}{w(B_0)^\kappa}
  \|f\|_{L^{p,\kappa}(w)}^p
  \sum_{k=1}^\infty w(2^{k+1}B_0)^{\kappa-1}.
\]

Now we use the doubling properties of $w$. Since $w\in A_p\subset A_\infty$, there exists $\delta>0$ such that
\[
  w(2^{k+1}B_0) \le C 2^{k n \delta} w(B_0)
  \quad\text{for all }k\ge1
\]
Hence
\[
  w(2^{k+1}B_0)^{\kappa-1}
  \le C^{\kappa-1} 2^{k n \delta (\kappa-1)} w(B_0)^{\kappa-1}.
\]
Since $\kappa<1$, the exponent $\kappa-1<0$, and thus
\[
  \sum_{k=1}^\infty w(2^{k+1}B_0)^{\kappa-1}
  \lesssim w(B_0)^{\kappa-1}.
\]
It follows that
\[
  I_2 \lesssim \frac{w(B_0)}{w(B_0)^\kappa}
  \|f\|_{L^{p,\kappa}(w)}^p\, w(B_0)^{\kappa-1}
  = \|f\|_{L^{p,\kappa}(w)}^p.
\]

\medskip\noindent
\emph{Step 3: Conclusion.}
Combining the estimates for $I_1$ and $I_2$,
\[
  I \lesssim \|f\|_{L^{p,\kappa}(w)}^p,
\]
with a constant independent of $B_0$ and $f$. Taking the supremum over $B_0$ yields
\[
  \|\mathcal{A}_{\mathcal{S}} f\|_{L^{p,\kappa}(w)}^p
  \lesssim \|f\|_{L^{p,\kappa}(w)}^p,
\]
which is equivalent to the desired estimate.
\end{proof}

\begin{remark}
The restriction $\kappa<1$ is essential in the argument, since it guarantees the geometric decay of $w(2^{k+1}B_0)^{\kappa-1}$ as $k\to\infty$. The case $\kappa=1$ corresponds to limiting spaces related to BMO and requires different techniques.
\end{remark}

\subsection{Boundeness of Calder\'on--Zygmund operators on weighetd Morrey spaces}Combining this pointwise domination \eqref{eq:sparse_dom} with Theorem~\ref{thm:sparseMorrey}, we can immediately obatin the boundedness Calderon-Zygmund operator $T$ on weighted Morrey spaces.

\begin{coro}
Let $1<p<\infty$, $0\le\kappa<1$, and $w\in A_p$. If $T$ admits a sparse domination of the form
\[
 |Tf(x)| \lesssim \sum_{j=1}^N \mathcal{A}_{\mathcal{S}_j} f(x),
\]
then $T$ is bounded on $L^{p,\kappa}(w)$:
\[
 \|Tf\|_{L^{p,\kappa}(w)}
 \lesssim \|f\|_{L^{p,\kappa}(w)}.
\]
\end{coro}

\begin{proof}
For each ball $B_0$,
\[
 \int_{B_0} |Tf(x)|^p w(x)\,dx 
 \lesssim 
 \sum_{j=1}^N \int_{B_0} |\mathcal{A}_{\mathcal{S}_j}f(x)|^p w(x)\,dx.
\]
Divide by $w(B_0)^\kappa$, take the supremum in $B_0$, and apply Theorem~\ref{thm:sparseMorrey} to each $\mathcal{S}_j$.
\end{proof}

This shows how sparse operators act as a bridge between Euclidean weighted $L^p$ theory and weighted Morrey theory. The same idea extends to fractional integral operators, multilinear operators, and certain non-standard singular integrals, provided suitable sparse domination results are available.

\section{Extensions and further directions}

We conclude with a few remarks on possible extensions of Theorem~\ref{thm:sparseMorrey}.

\subsection{Other Morrey-type and generalized spaces}

The Komori--Shirai spaces $L^{p,\kappa}(w)$ are only one example from a larger family of Morrey-type spaces:

\medskip

\subsubsection{Generalized Morrey spaces} These spaces extend the classical Morrey scale by replacing the factor $|B|^{-\lambda/p}$ (or $w(B)^{-\kappa}$ in the weighted setting) with a control function. Let $\phi : (0,\infty) \to (0,\infty)$ be a positive, increasing function. Following the framework introduced by Mizuhara~\cite{Mizuhara1991} and later developed systematically by Nakai~\cite{Nakai1994}, the generalized Morrey space $L^{p,\phi}(\mathbb{R}^{n})$ consists of all locally integrable functions $f$ such that 
\[
\|f\|_{L^{p,\phi}}
:= \sup_{B(x,r)}
\frac{1}{\phi(r)}
\left( \int_{B(x,r)} |f(y)|^{p}\, dy \right)^{1/p}
< \infty,
\]
where the supremum is taken over all balls $B(x,r) \subset \mathbb{R}^{n}$. 
Note that using different choices of $\phi$ we recover a variety of Morrey-type spaces. For instance, the classical Morrey space corresponds to $\phi(r) = r^{\lambda/p - n/p}$. Note that numerous variants and weighted versions of these spaces have since been studied, see Sawano~\cite{SawanoBook} and Gunawan et al.~\cite{GunawanInclusion}.

\medskip

\subsubsection{Orlicz--Morrey spaces} These form another natural extension of the classical Morrey scale in which the local $L^{p}$ norm is replaced by an Orlicz norm. Let $\Phi$ be a Young function and let $\phi : (0,\infty)\to (0,\infty)$ be a positive control function. Following the framework introduced by Nakai~\cite{Nakai2004OM} and later refined by Sawano and collaborators, the Orlicz--Morrey space $L^{\Phi,\phi}(\mathbb{R}^{n})$ consists of all measurable functions $f$ for which
\[
\|f\|_{L^{\Phi,\phi}}
:= \sup_{B(x,r)}
\frac{1}{\phi(r)}
\|f\|_{L^{\Phi}(B(x,r))}
<\infty.
\]
Here $\|f\|_{L^{\Phi}(B)}$ denotes the Orlicz Luxemburg norm associated to $\Phi$ on the ball $B$, and the function $\phi$ controls the scaling behaviour. Note that choosing $\Phi(t)=t^{p}$ recovers the generalized Morrey space $L^{p,\phi}$. Orlicz-Morrey spaces have been studied extensively in recent years, especially in the context of boundedness of maximal operators, singular integrals, and commutators, see~\cite{Nakai2004OM, GuliyevOM, SawanoOM}.

\subsubsection{Variable exponent Morrey spaces} These spaces extend the classical Morrey framework by allowing the integrability exponent to vary from point to point. Let $p:\mathbb{R}^{n}\to [1,\infty)$ be a measurable function satisfying suitable log-HAlder continuity conditions, and let $\lambda:\mathbb{R}^{n}\to [0,n)$ be a measurable function. These spaces were introduced by Almeida and Samko~\cite{AlmeidaSamko2009} and further developed by Ho~\cite{Ho2014}, the variable exponent Morrey space $L^{p(\cdot),\lambda(\cdot)}(\mathbb{R}^{n})$ consists of all measurable functions $f$ such that
\[
\|f\|_{L^{p(\cdot),\lambda(\cdot)}}
:= \sup_{B(x,r)}
r^{-\lambda(x)/p(x)}
\| f \chi_{B(x,r)} \|_{L^{p(\cdot)}}
< \infty,
\]
where $\|\cdot\|_{L^{p(\cdot)}}$ denotes the variable exponent Lebesgue norm. When $p(\cdot)\equiv p$ and $\lambda(\cdot)\equiv \lambda$ are constant functions, this reduces to the classical Morrey space $L^{p,\lambda}(\mathbb{R}^{n})$. Variable exponent Morrey spaces have been used to study maximal operators, fractional integrals, and singular integrals in settings where spatial inhomogeneity affects both the local integrability and decay conditions, see~\cite{AlmeidaSamko2009,Ho2014,HastoDiening} for detailed treatments.

We would like to remark that in many of these settings, sparse domination is also available and the proof scheme used in Section 3 can be adapted. The key ingredients are similar incluidng a decomposition into local and global parts relative to a testing ball $B_0$;
and $L^p$-type boundedness of sparse operators in the underlying weighted space, a doubling-type property that controls $w(\lambda B_0)$ in terms of $w(B_0)$, combined with the condition that the Morrey parameter (such as $\kappa$) is strictly less than $1$ to ensure geometric decay in the far annuli to obtain the boundedness of the appropriate sparse operators on these spaces.




\subsection{Commuatator boundedness on weighted Morrey spaces} Another application of sparse domination is to show the boundedness of Calder\'on-Zygmund commutators on weighted Morrey spaces. We can achieve this using the sparse domination of CZO commutator's by the Bloom sparse operators. Let us recall that given a Calder\'on--Zygmund operator $T$ and $b\in \mathrm{BMO}$, one has the pointwise bound
\begin{equation}\label{eq:sparse-commutator}
|[b,T]f(x)|
\;\lesssim\;
\sum_{j=1}^{N} \mathcal{T}_{\mathcal{S}_{j}}f(x) +\mathcal{T}^{*}_{\mathcal{S}_{j}} f(x)
\end{equation}
where each $\mathcal{S}_{j}$ is a sparse family of cubes and
\[
\mathcal{T}_{\mathcal{S}}f(x)
=
\sum_{Q\in\mathcal{S}}
|b(x)-b_{Q}|
\left( \frac{1}{|Q|}\int_{Q} |f(y)| \, dy \right)
\mathbf{1}_{Q}(x),
\]
\[
\mathcal{T}^{*}_{\mathcal{S}}f(x)
=
\sum_{Q\in\mathcal{S}}
\left( \frac{1}{|Q|}\int_{Q} |b(y)-b_{Q}|\,|f(y)| \, dy \right)
\mathbf{1}_{Q}(x).
\]
The operator $\mathcal{T}_{\mathcal{S}}$ is the ``Bloom sparse operator,'' 
and $\mathcal{T}^{*}_{\mathcal{S}}$ is its adjoint. The domination \eqref{eq:sparse-commutator} was established independently by 
Lerner--Ombrosi--Rivera-R\'ios~\cite{LernerOmbrosiRivera} and 
Holmes--Lacey--Wick~\cite{HolmesLaceyWick}, and forms the modern basis for Bloom-type 
inequalities for commutators.  
\begin{theorem}\label{thm:bloom-sparse-morrey}
Let $1<p<\infty$, $0\le \kappa<1$, and let $w\in A_{p}$. Let $b\in \mathrm{BMO}(\mathbb{R}^{n})$ and let $\mathcal{S}$ be a sparse family of cubes. Then the operator $\mathcal{T}_{\mathcal{S}}$ defined above  is bounded on the weighted Morrey space $L^{p,\kappa}(w)$; that is, there exists a constant $C>0$, depending only on $n,p,\kappa$ and $[w]_{A_{p}}$, such that
\[
\|\mathcal{T}_{\mathcal{S}} f\|_{L^{p,\kappa}(w)}
\le C\,\|b\|_{\mathrm{BMO}}\,\|f\|_{L^{p,\kappa}(w)}
\qquad\text{for all } f\in L^{p,\kappa}(w).
\]
\end{theorem}

\begin{remark}
We briefly indicate how one may prove Theorem~\ref{thm:bloom-sparse-morrey}, but we do not give the full argument here. The key idea is the $L^{p}(w)$ boundedness of Bloom-type sparse operators, proved by a sparse duality argument using the weighted John--Nirenberg inequality and a Carleson embedding estimate.  To obtain the Morrey bound, one fixes a ball and decomposes $f$ into local and global parts; the local term is controlled by the $L^{p}(w)$ estimate, while sparsity ensures that the global contribution involves only cubes containing the ball.  
Hence using the BMO oscillation, Morrey control on dyadic annuli, and the doubling and reverse H\"older properties of $A_{p}$ weights, one obtains a summable geometric series, yielding the desired estimate.
\end{remark}

As an application, one can combine Theorem~\ref{thm:bloom-sparse-morrey} with known sparse domination result for commutators of Calder\'on--Zygmund operators given by \eqref{eq:sparse-commutator} to obtain the boundedness of these operators on Morrey spaces. Before stating the boundedness result, let us recall the definition for commutators on the Morrey space $L^{p,\kappa}(w)$.

\begin{definition} Let $1<p<\infty$, $0\le \kappa<1$, and $w\in A_p$. Let us denote by $X = \overline{C_c^\infty(\mathbb{R}^n)}^{\,L^{p,\kappa}(w)}$. Then for $f\in C_c^\infty(\mathbb{R}^n)$, we define commutator as
\[
[b,T]f := b\,Tf - T(bf),
\].
For a general $f\in L^{p,\kappa}(w)$, the commutator $[b,T]f$ is defined as the canonical image in $X^{**}$ of the operator initially defined on $X$. Here we use biduality to see $X^{**}=L^{p,\kappa}(w)$, which yields us a well-defined element of $L^{p,\kappa}(w)$; see \cite[Section~10.5.1]{SawanoBook} for more details.
\end{definition}

\begin{coro}\label{cor:commutator-morrey}
Let $T$ be a Calder\'on--Zygmund operator on $\mathbb{R}^{n}$ and $b\in \mathrm{BMO}(\mathbb{R}^{n})$. Assume $1<p<\infty$, $0\le \kappa<1$, and $w\in A_{p}$. Then the commutator $[b,T]$ described above extends to a bounded operator on $L^{p,\kappa}(w)$, that is,
\[
\|[b,T]f\|_{L^{p,\kappa}(w)}
\le C\,\|b\|_{\mathrm{BMO}}\,\|f\|_{L^{p,\kappa}(w)}
\qquad\text{for all } f\in L^{p,\kappa}(w),
\]
with a constant $C$ independent of $f$.
\end{coro}

In summary, sparse domination provides a elegant route to Morrey-type inequalities. The explicit boundedness of sparse operators on weighted Morrey spaces, as established here, serves as a model result and a convenient building block for more sophisticated applications. In particular, it offers a unified framework that seamlessly extends to commutators, multilinear operators, and potential two-weight Bloom settings. We expect this approach to be adaptable to broader Morrey-type scales, including generalized, 
Orlicz-Morrey, and variable-exponent Morrey spaces as described earlier.



\textbf{Acknowledgement.} The author is grateful to the referee for helpful comments.

\bibliographystyle{amsalpha}

\begin{thebibliography}{A}



\bibitem{Adams}
D. R. ~Adams, \emph{A note on Riesz potentials}, Duke Math. J. {\bf 42} (1975), 765--778.


\bibitem{AlmeidaSamko2009}
A.~Almeida and S.~Samko,
\emph{Morrey spaces with variable exponent},
J. Funct. Spaces Appl. 7 (2009), no.~2, 143-167.

\bibitem{BaileyEtAl}
J.~Bailey, A.~B\'en\'eteau, F.~Bernicot, et al.,
\emph{Quadratic sparse domination and weighted estimates for operators with non-smooth kernels},
J. Geom. Anal. \textbf{33} (2023), no.~4, Paper 128.%

\bibitem{Barron}
A.~Barron,
\emph{Sparse domination and the strong maximal function},
Adv. Math. \textbf{350} (2019), 324-350.%

\bibitem{BernicotBernicot}
F.~Bernicot, D.~Cruz-Uribe, and J.~Martell,
\emph{Multilinear Calder\'on--Zygmund theory, dyadic cubes, and sparse domination},
Indiana Univ.\ Math.\ J.\ \textbf{67} (2018), no.~5, 1747-1784.

\bibitem{BernicotFreyPetermichl}
F.~Bernicot, D.~Frey, and S.~Petermichl,
\emph{Sharp weighted estimates for the conical square function},
Anal.\ PDE \textbf{9} (2016), no.~5, 1079-1113.

\bibitem{ChenDingWang}
Y.~Chen, Y.~Ding, X.~Wang,
\emph{Compactness of commutators for singular integrals on Morrey spaces},
Canad. J. Math. \textbf{63} (2011), no.~5, 1087-1109.%

\bibitem{CondeAlonsoRey}
J.~M.~Conde-Alonso and G.~Rey,
\emph{A pointwise estimate for positive dyadic shifts and some applications},
Math.\ Ann.\ \textbf{365} (2016), no.~3-4, 1111-135.

\bibitem{FM}
F. ~Chiarenza and M. ~Frasca, \emph{Morrey spaces and Hardy-Littlewood maximal function},
Rend. Mat. Appl. \textbf{7}(1987), 273-279.


\bibitem{jm}
J. ~Duoandikoetxea, M. ~Rosenthal,
\emph{Boundedness properties in a family of weighted Morrey spaces with emphasis on power weights},
Journ. Func. Anal., {\bf 8} (2020), 108687.

\bibitem{rjem}
R. ~Gong, J. ~Li, E. ~Pozzi, M. N. ~Vempati,
\emph{Commutators on weighted Morrey spaces on spaces of homogeneous type},
Anal. Geom. Metr. Spac. {\bf 8} (2021), 305-334.

\bibitem{gvqx}
R. ~Gong, M. ~N. ~Vempati, Q. ~Wu, P. ~Xie,
\emph{Boundedness and compactness of Cauchy-type integral commutator on weighted Morrey spaces},
Journ.  Aust. Math. Soc., {\bf 113} (1), 36-56, (2022).

\bibitem{GuliyevOM}
V.~S.~Guliyev,
\emph{Boundedness of the maximal, potential and singular operators in the generalized Orlicz--Morrey spaces},
J. Inequal. Appl. 2009, Article ID 503948, 20~pp.

\bibitem{GunawanBessel}
H.~Gunawan, E.~N.~D.~Eridani,
\emph{The boundedness of Bessel--Riesz operators on generalized Morrey spaces},
Austral. J. Math. Anal. Appl. \textbf{13} (2016), no.~1, Art.~9, 14~pp.%

\bibitem{GunawanInclusion}
H.~Gunawan, et al.,
\emph{Inclusion properties of generalized Morrey spaces},
Math. Nachr. \textbf{290} (2017), no.~2--3, 332-340.%

\bibitem{HastoDiening}
P.~H\"ast\"o and L.~Diening,
\emph{Variable exponent function spaces},
in: ``Lectures on Real Analysis'', Springer Lecture Notes.

\bibitem{HolmesLaceyWick}
I.~Holmes, M.~T.~Lacey, and B.~D.~Wick,
\emph{Commutators in the two-weight setting},
J.\ Geom.\ Anal.\ \textbf{28} (2018), no.~4, 3248-3278.

\bibitem{Ho2014}
K.-P.~Ho,
\emph{Some properties of variable exponent Morrey spaces},
Math. Inequal. Appl. 17 (2014), no.~1, 279-289.

\bibitem{HytTalOscillatory}
T.~HytAnen and J.~M.~Martikainen,
\emph{Non-homogeneous $Tb$ theorem and random dyadic cubes on metric spaces},
J.\ Geom.\ Anal.\ \textbf{22} (2012), 1071-1107.


\bibitem{KomoriShirai}
Y.~Komori, S.~Shirai,
\emph{Weighted Morrey spaces and a singular integral operator},
Math. Nachr. \textbf{282} (2009), no.~2, 219-231.%

\bibitem{LaceySparse}
M.~T.~Lacey,
\emph{Sparse bounds for oscillatory and random singular integrals},
New York J. Math. \textbf{23} (2017), 1190-1223.%


\bibitem{LernerLocal}
A.~K.~Lerner,
\emph{A pointwise estimate for the local sharp maximal function with applications to singular integrals},
Bull. Lond. Math. Soc. \textbf{42} (2010), no.~5, 843-856.%


\bibitem{LernerA2}
A.~K.~Lerner,
\emph{A simple proof of the $A_2$ conjecture},
Int. Math. Res. Not. IMRN 2013, no.~14, 3159-3170.%

\bibitem{LernerOF}
A.~K.~Lerner,
\emph{Operator-free sparse domination},
preprint and related expositions, Helda institutional repository, 2022.%

\bibitem{akl}
A. ~K. ~Lerner, \emph{A note on the maximal operator on weighted Morrey
spaces}, Anal. Math., {\bf 49} (2023), 1073-1086.



\bibitem{LernerVariation}
A.~K.~Lerner and F.~Nazarov,
\emph{Intuitive sparse bounds for variational Carleson operators},
in: Harmonic Analysis, Partial Differential Equations, and Applications,
Springer (2017), 163-178.


\bibitem{LernerOmbrosi}
A.~K.~Lerner, S.~Ombrosi,
\emph{Some remarks on the pointwise sparse domination},
J. Geom. Anal. \textbf{30} (2020), no.~1, 101-134.%

\bibitem{LernerOmbrosiRivera}
A.~K.~Lerner, S.~Ombrosi, and I.~P.~Rivera-R\'ios,
\emph{On pointwise and weighted estimates for commutators of multilinear singular integral operators},
Ann.\ Sc.\ Norm.\ Super.\ Pisa Cl.\ Sci.\ (5) \textbf{17} (2017), no.~2, 659-678.

\bibitem{mnv}
M. N. Vempati,
\emph{Commutators of maximal Operators on weighted Morrey spaces,}
Int. Trans. Spec. Func., (2025) 1-18.

\bibitem{Morrey}
C.~B.~Morrey,
\emph{On the solutions of quasi-linear elliptic partial differential equations},
Trans. Amer. Math. Soc. \textbf{43} (1938), 126-166.%

\bibitem{Mizuhara1991}
T.~Mizuhara,
\emph{Boundedness of some classical operators on generalized Morrey spaces},
in: Harmonic Analysis (Sendai, 1990),
Lecture Notes in Mathematics, vol.~1494,
Springer, Berlin, 1991, 183-189.

\bibitem{NakamuraSawano}
S.~Nakamura, Y.~Sawano, H.~Tanaka
\emph{Weighted local Morrey spaces},
Ann. Acad. Sci. Fenn. Math. \textbf{45} (2020), no.~1, 67-92.%

\bibitem{NakamuraSawanoTanaka2018}
S.~Nakamura, Y.~Sawano, H.~Tanaka,
\emph{The fractional operators on weighted Morrey spaces}, J. Geom. Anal., \textbf{28} (2018), 1502-1524.

\bibitem{Nakai1994}
E.~Nakai,
\emph{Hardy--Littlewood maximal operator, singular integral operators and Riesz potentials on generalized Morrey spaces},
Math.\ Nachr.\ 166 (1994), 95-103.

\bibitem{Nakai2004OM}
E.~Nakai,
\emph{Orlicz--Morrey spaces and the Hardy--Littlewood maximal function},
Studia Math. 167 (2005), no.~3, 295-314.

\bibitem{Olsen}
P. ~A. ~Olsen,
\emph{Fractional integration, Morrey spaces and a Schr\"odinger equation},
Comm. Part. Diff. Equa. \textbf{20} (1995), 2005-2055.

\bibitem{Ramadana}
Y.~Ramadana, et al.,
\emph{Two-weighted estimates for some sublinear operators on generalized weighted Morrey spaces},
preprint, arXiv:2406.05435, 2024.%


\bibitem{SawanoBook}
Y.~Sawano, G.~ Di Fazio, D. I. ~Hakim
\emph{Morrey Spaces--Introduction
and applications to integral operators and PDE's}, Vol. II, Monogr.
Res. Notes Math., CRC Press, Boca Raton, FL, (2020), xviii+409.

\bibitem{SawanoNote}
Y.~Sawano,
\emph{A thought on generalized Morrey spaces},
J. Indian Math. Soc. (N.S.) \textbf{86} (2019), no.~1--2, 1-26.%

\bibitem{SawanoOM}
Y.~Sawano,
\emph{Orlicz--Morrey spaces and related operators},
in: ``Function Spaces and Inequalities'', Springer Proc. Math. Stat. 2014, pp.~169-186.

\bibitem{Tanaka}
H.~Tanaka,
\emph{Morrey spaces and fractional operators},
J. Aust. Math. Soc. \textbf{86} (2009), 417-434.%

\bibitem{TanWangXue}
J.~Tan, J.~Wang, Q.~Xue,
\emph{Boundedness of a class of multilinear operators and their iterated commutators on Morrey--Banach function spaces},
preprint, arXiv:2502.08456, 2025.%

\bibitem{vsg}
V. S. Guliyev,
\emph{Generalized weighted Morrey spaces and higher order commutators of sublinear operators},
Eur. Math. Journ. {\bf 3} (2012), 33-61.

\bibitem{WangZheng}
G.~Wang, S.~Zheng,
\emph{Boundedness on generalized Morrey spaces for the Schr\"odinger operator with potential in a reverse H\"older class},
Electron. J. Differential Equations 2023, No.~67, 1-14.%

\end{thebibliography}

\end{document}